\documentclass[a4paper,12pt]{article} 
\usepackage{graphicx}
\usepackage{amsmath,amssymb,textcomp,graphicx,amsthm} 
\usepackage{epsfig}
\usepackage{color}
\usepackage[english]{babel} 
\usepackage[alphabetic]{amsrefs}
 \usepackage{verbatim}
\usepackage[latin1]{inputenc}  

\numberwithin{equation}{section}

\newcommand{\e}{\varepsilon}

\newcommand{\tr}{\operatorname{tr}}

\newcommand{\R}{\mathbb{R}}
\newcommand{\N}{\mathbb{N}}

\newtheorem{de}{Definition}
\newtheorem{lem}{Lemma}
\newtheorem{prop}{Proposition}
\newtheorem{thm}{Theorem}
\newtheorem{cor}{Corollary}

\theoremstyle{remark}
\newtheorem{rem}{Remark}

\newcommand{\lap}{\Delta}

\renewcommand{\div}{\operatorname{div}}
\newcommand{\osc}{\operatorname{osc}}

\title{Lipschitz regularity for a homogeneous doubly nonlinear PDE}

\author{Ryan Hynd\footnote{rhynd@math.upenn.edu, Department of Mathematics, University of Pennsylvania, Philadelphia, Pennsylvania 19104.  Partially supported by NSF grant DMS-1301628.}\; and Erik Lindgren\footnote{erik.lindgren@math.uu.se, Department of Mathematics, Uppsala University, Box 480
751 06, Uppsala, Sweden. Supported by the Swedish Research Council, grant no. 2012-3124 and 2017-03736. Part of this work was carried out when the second author was visiting University of Pennsylvania. The math department and its facilities are kindly acknowledged.}}  
\begin{document}\maketitle

\begin{abstract}\noindent   We study the doubly nonlinear PDE
$$
|\partial_t u|^{p-2}\,\partial_t u-\div(|\nabla u|^{p-2}\nabla u)=0.
$$
This equation arises in the study of extremals of Poincar\'e inequalities in Sobolev spaces. We prove spatial Lipschitz continuity and H\"older continuity in time of order $(p-1)/p$ for viscosity solutions. As an application of our estimates, we obtain pointwise control of the large time behavior of solutions.
\end{abstract}

\maketitle
\section{Introduction} We the study the local regularity of viscosity solutions of the doubly nonlinear parabolic equation
\begin{equation}\label{eq:art}
|\partial_t u|^{p-2}\,\partial_t u-\lap_p u=0 \, ,
\end{equation}
for $p\in [2,\infty)$. Here $\lap_p$ is the $p$-Laplace operator
$$
\lap_p u = \div(|\nabla u|^{p-2}\nabla u),
$$
the first variation of the functional
$$
u\mapsto \int_\Omega |\nabla u|^p dx.
$$
The first occurence of \eqref{eq:art} that we have found is in a footnote in \cite{KL96}. Our interest in \eqref{eq:art} relies on the connection to the eigenvalue problem for the $p$-Laplacian. See our previous work \cite{HL16}, \cite{HL17} and also Theorem \ref{LargeTthm} in Section \ref{sec:largetime}. This eigenvalue problem amounts to studying extremals of the Rayleigh quotient
\begin{equation}\label{rayleigh}
\lambda_{p} = \inf_{u\in W_0^{1,p}(\Omega)\setminus \{0\}}\frac{\displaystyle\int_\Omega |\nabla u|^p dx}{\displaystyle\int_\Omega |u|^p dx}.
\end{equation}
Here $\Omega\subset \R^n$ is a bounded and open set. Extremals are often called \emph{ground states}. This extremal problem is naturally equivalent to finding the optimal constant in the Poincar\'e inequality for $W^{1,p}_0(\Omega)$.

\subsection{Main results} The first of our results is spatial Lipschitz continuity and H\"older continuity in time of order $(p-1)/p$. These are proved using Ishii-Lions' method, introduced in \cite{IL90}. To the best of our knowledge, this is the first pointwise regularity result for this equation. In order to state our first theorem we introduce the notation $$Q_r(x_0,t_0)=B_r(x_0)\times (t_0-r^\frac{p}{p-1},t_0].$$

\begin{thm}\label{thm:main} Let $p\geq 2$, $\Omega$ be a bounded and open set in $\R^n$ and $I$ a bounded interval. Suppose $u$ is a viscosity solution of \eqref{eq:art} in $\Omega  \times I$. Then 

$$
|u(x,t)-u(y,s)|\leq \frac{C(n,p)}{R}\|u\|_{L^\infty(Q_{2R}(x_0,t_0))}\left(|x-y| + |t-s|^\frac{p-1}{p} \right).
$$
for any $(x,t),\;(y,s)\in Q_{R/2}(x_0,t_0)$ and for every $R$, $x_0$ and $t_0$ such that $Q_{2R}(x_0,t_0)\Subset\Omega\times I$.
\end{thm}

Our second result concerns the large time behavior of solutions. This was investigated in our previous work \cite{HL16}. In particular, there exists a ground state $w$ such that $$e^{\lambda_p^\frac{1}{p-1} t}u(x,t)\to w$$ in $W_0^{1,p}(\Omega)$, when $u$ solves
\begin{equation}\label{pParabolic}
\begin{cases}
|\partial_t u|^{p-2}\,\partial_t u=\Delta_p u,\quad  \Omega\times (0,\infty), \\
u=0,\quad 
 \partial \Omega\times [0,\infty),\\
 u=g,\quad 
 \Omega\times\{0\},
\end{cases}
\end{equation}
see Theorem \ref{LargeTthm}. As a consequence of this and Theorem \ref{thm:main}, we obtain that this convergence is uniform.
\begin{thm}\label{thm:largetime} Let $p\in [2,\infty)$, $\Omega$ be a bounded and regular\footnote{Regular in the sense that any groundstate is continuous up to the boundary. This is true for instance if  $\partial\Omega$ is Lipschitz.} domain and assume that $g\in W_0^{1,p}(\Omega)\cap C(\overline\Omega)$ satisfies $|g|\leq \phi$, where $\phi$ is a ground state. If $u$ is a  viscosity solution of \eqref{pParabolic}, then there is a ground state $w$ such that 
$$e^{\lambda_p^\frac{1}{p-1} t}u(x,t)\to w,$$
uniformly in $\overline\Omega$.
\end{thm} 
We do not expect the estimates in Theorem \ref{thm:main} to be sharp. In our opinion, solutions are likely to be at least continuously differentiable in space, even though we are unable to verify this at the moment. Concerning time regularity, it may be a very delicate task to obtain any higher H\"older exponent. See the next section for a comparison with related equations.





\subsection{Known results}
Doubly nonlinear equations such as \eqref{eq:art} have mostly been studied from a functional analytic point of view. See for instance \cite{MRS} and \cite{Ste}. However, the pointwise properties and in particular the regularity theory has not been developed. Needless to say, the nonlinearity in the time derivative presents a genuine challenge. A related result can be found in \cite{HL16b}, where H\"older estimates for some H\"older exponent are proved for the doubly nonlinear non-local equation
$$
|\partial_t u|^{p-2}\,\partial_t u+(-\Delta_p)^s u=0.
$$
The large time behavior of solutions has a natural connection to the Poincar\'e inequality in the fractional Sobolev space $W^{s,p}$, the non-local counterpart of  \eqref{rayleigh}.

The related $p$-parabolic equation 
$$\partial_t u-\Delta_p u=0,$$
has been given vast attention the past 30 years. In contrast to \eqref{eq:art}, this equation is not homogeneous. Due to the the linearity in the time derivative, the notion of weak solutions turns out to be more useful than for \eqref{eq:art}. We refer to \cite{DiBbook} for an overview of the regularity theory. The best local regularity known is spatial $C^{1,\alpha}$-regularity for some $\alpha>0$ and $C^{1/2}$-regularity in time. Neither of these exponents are known to be sharp. Due to the explicit solution
$$
u=t\cdot n-\frac{p-1}{p}|x|^\frac{p}{p-1}, 
$$
where $n$ is the dimension, it is clear that solutions cannot be better than $C^{1,1/(p-1)}$ in space. 

Recently, Ishii-Lions' method has been used for equations involving the $p$-Laplacian. In \cite{IJS16}, the authors used it to study the regularity of solutions of 
$$
\partial_t u = |\nabla u|^\kappa\Delta_p u,\quad \kappa \in (1-p,\infty).
$$
In the recent papers \cite{APR17} and \cite{AP18} it is used for the equations
$$
|\nabla u|^{2-p}\Delta_p u = f,\quad \partial_t u-|\nabla u|^{2-p}\Delta_p u = f.
$$

\subsection{The idea of the proof} For many elliptic or parabolic equations including \eqref{eq:art}, it is possible to prove a comparison principle. When working with viscosity solutions, this is usually accomplished by doubling the variables. This amounts to ruling out that
$$
\sup_{x,y}\,\left(u(x)-v(y)-\phi(|x-y|)\right)>0
$$
when $u$ is a subsolution, $v$ is a supersolution, $u\leq v$ on the boundary and $\phi$ is appropriately chosen. For uniformly elliptic equations the choice $\phi(r)=r^2$ is suitable to prove a comparison principle.

It turns out that a similar approach can also give continuity estimates. This was first done in \cite{IL90}. A spatial continuity estimate of order $\phi(r)$ for a solution $u$ of \eqref{eq:art} is saying that 
$$
\sup_{x,y,t}\,\left(u(x,t)-u(y,t)-\phi(|x-y|)\right)\leq 0.
$$
In order to prove this, we assume towards a contradiction that
$$
\sup_{x,y,t}\,\left(u(x,t)-u(y,t)-\phi(|x-y|)\right)> 0.
$$
In this paper, we work with the choices $\phi(r)\approx r|\ln r|$ and $\phi(r)\approx r$. This gives a log-Lipschitz and a Lipschitz estimate in the spatial variables. In contrast to the case $\phi(r)=r^2$, $\phi$ is here chosen so that it is strictly concave.  The spatial regularity can be used to construct a suitable supersolution that yields the desired time regularity.

\subsection{Plan of the paper}
The plan of the paper is as follows. In Section \ref{sec:visc}, we introduce some notation and the notion of viscosity solutions.  This is followed by Section \ref{sec:loglip}, where we prove log-Lipschitz continuity in space. In Section \ref{sec:lip}, we improve this to Lipschitz continuity. This result is then used in Section \ref{sec:time}, where we prove the corresponding H\"older regularity in time. We combine these results in Section \ref{sec:main}, where we prove our main regularity theorem. Finally, in Section \ref{sec:largetime}, we study the large time behavior.


\section{Notation and prerequisites}\label{sec:visc} 
Throughout the paper, we will use the notation $$Q_r(x_0,t_0)=B_r(x_0)\times (t_0-r^\frac{p}{p-1},t_0]$$ and $Q_r=B_r(0)\times (-r^\frac{p}{p-1},0]$. These are cylinders reflecting the natural scaling of solutions to \eqref{eq:art}. We will also use
the matrix norm
$$
\|X\|=\sup_{\|\xi\|\leq 1}|X\xi|.
$$
In addition, we will, for any subset of $Q\subset \R^{n+1}$, use the notation
$$
\osc_{Q} u=\sup_{Q} u-\inf_{Q} u.
$$
For completeness we include the definition of viscosity solutions:
\begin{de}
Let $\Omega \in \R^n$ be an open set and $I\in \R$ be a bounded interval. A function which is upper semicontinuous in ${\Omega\times I}$ is a \emph{subsolution} of 
$$  
|\partial_t v|^{p-2}\partial_t v-\lap_p v\,\leq 0 \text{ in $\Omega \times I$}
$$
if the following holds: whenever $(x_0,t_0)\in \Omega\times I$ and $\phi\in C^{2,1}_{x,t}({B_r(x_0)}\times (t_0-r,t_0])$  for some $r>0$ are such that
$$
\phi(x_0,t_0)=v(x_0,t_0), \quad \phi(x,t)\geq v(x,t) \text{ for $(x,t)\in B_r(x_0)\times (t_0-r,t_0]$}
$$
then
$$
|\partial_t \phi(x_0,t_0) |^{p-2}\partial_t \phi(x_0,t_0)-\lap_p\phi\, (x_0,t_0)\leq 0.
$$
A \emph{supersolution} is defined similarly and a \emph{solution} is a function which is both a sub- and a supersolution.
\end{de}
\begin{rem} The notion of viscosity solutions may also be formulated in terms of so called jets: $v$ is a viscosity subsolution in $\Omega \times I$ if $(\alpha,a,X)\in \overline{\mathcal{P}}^{2,+}_Q v\,(x_0,t_0)$ for $(x_0,t_0)\in Q$ for some cylinder $Q\subset \Omega\times I$ implies
$$
|\alpha|^{p-2}\alpha-|a|^{p-2}\tr (L(a)X)\leq 0, \quad L(a)=I+(p-2)\frac{a\otimes a}{|a|^2}.
$$
See \cite{CIL92} and \cite{DFO11} for further reading. Here and throughout the paper we will use the notation used in \cite{DFO11}.
\end{rem}
In \cite{HL16}, the following comparison principle is mentioned. The proof of this result is identical to for instance the proof of Theorem 4.7 of \cite{JLM01}.
\begin{prop}\label{UsualComparison}
Assume $v\in USC(\overline{\Omega}\times[0,T))$ and $w\in LSC(\overline{\Omega}\times[0,T))$. Suppose the inequality
$$
|\partial_t v|^{p-2}\partial_t v-\Delta_pv\le 0\le |\partial_t w|^{p-2}\partial_t w-\Delta_pw, \quad \Omega\times(0,T)
$$
holds in the sense of viscosity solutions and $v(x,t)\le w(x,t)$ for $(x,t)\in \partial\Omega\times[0,T)$ and for $(x,t)\in \Omega\times\{0\}$. Then 
$$
v\le w
$$
in $\Omega\times(0,T).$
\end{prop}

\section{Log-Lipschitz regularity }\label{sec:loglip}
We start with a technical calculus result.
\begin{lem}\label{lem:phi1} Let 
$$
\phi(r)=\begin{cases}  -r\ln r, & 0<r<e^{-1}\\e^{-1}, &r\geq e^{-1}.\end{cases}
$$
Then $\phi(r)<1/8$ implies
$$
r<e^{-2},\quad  \phi(r)=-r\ln r,\quad  \phi'(r)\geq 1,\quad |\phi''(r)|\leq \phi'(r)/r.$$
\end{lem}
\begin{proof}
First we note that $\phi$ is non-decreasing. Moreover, $$\phi(e^{-2})=2e^{-2}>1/8>\phi(r).$$ Therefore, $r<e^{-2}$, which also implies that $\phi(r)=-r\ln r$. In addition, 
$$\phi'(r)=-1-\ln r,
$$
is a non-increasing function and $\phi'(e^{-2})=1$. Therefore, $\phi'(r)\geq 1$.

Finally, 
$$
|\phi''(r)|=r^{-1}\leq \phi'(r)r^{-1}=-r^{-1}-r^{-1}\ln r
$$
since $r<e^{-2}$.
\end{proof}
\begin{prop}\label{prop:loglip} Suppose $u$ is a viscosity solution of \eqref{eq:art} in $Q_2$ such that $\osc_{Q_2} u\leq 1$. Then 
$$
u(t,x)-u(t,y)\leq A\phi (|x-y|)+\frac{B}{2}\left(|x|^2+|y|^2+t^2\right), 
$$
for $(x,t)\in Q_1$. Here 
$$
\phi(r)=\begin{cases} 0<-r\ln r, & r<e^{-1}\\e^{-1}, &r\geq e^{-1}\end{cases}
$$
and $A=A(n,p)$ and $B$ is universal.
\end{prop}
\begin{proof}
Let 
$$
\Phi(x,y,t)=u(x,t)-u(y,t)-A\phi(|x-y|)-\frac{B}{2}\left(x^2+y^2+t^2\right).
$$
In order to show the desired inequality, we assume towards a contradiction that $\Phi$ assumes a positive maximum at some $t\in [-1,0]$ and $x,y\in  \overline B_1$. Since $\Phi(x,y,t)>0$ we have
\begin{equation}
\label{eq:ABineq}
A\phi(|x-y|)+\frac{B}{2}\left(|x|^2+|y|^2+t^2\right)\leq |u(t,x)-u(t,y)|\leq 1.
\end{equation}
Therefore, we may choose $B=4$ ($B>2$ is enough), so that $x,y\in B_1$ and $t\in (-1,0]$. Let us introduce the notation 
$$
\hat \delta=\frac{x-y}{|x-y|},\quad  \delta =|x-y|.
$$	
By choosing $A>8$ we see that \eqref{eq:ABineq} combined with Lemma \ref{lem:phi1} implies
\begin{equation}\label{eq:delta1}
\delta<e^{-2},\quad  \phi(\delta)=-\delta\ln \delta,\quad  \phi'(\delta)\geq 1,\quad |\phi''(\delta)|\leq \phi'(\delta)/\delta.
\end{equation}
It also follows that
$$
u(x,t)-u(y,t)>0,
$$
implying that $\delta\neq  0$. \\

\noindent {\bf Step 1: Applying the Theorem of sums.}
From the parabolic theorem of sums (Theorem 8.3 in \cite{CIL92} and Theorem 9 in \cite{DFO11}), for any $\tau>0$ there are $X, Y\in S(n)$, $\alpha_1$ and $\alpha_2$ such that\footnote{$S(n)$ stands for symmetric $n\times n$ matrices}
\begin{align*}
(\alpha_1,A\phi'(\delta)\hat\delta,X)\in \overline{\mathcal{P}}_{Q_1}^{2,+}\left(u-\frac{B}{2}|\cdot |^2\right)\,(x,t),\\
 (-\alpha_2,A\phi'(\delta)\hat\delta,Y)\in \overline{\mathcal{P}}_{Q_1}^{2,-}\left(u+\frac{B}{2}|\cdot|^2\right)\,(y,t),
\end{align*}
$$
\alpha_1+\alpha_2 \geq Bt,
$$
\begin{equation}
\label{eq1:sum1}
-[\tau +||Z||]\begin{bmatrix}
I& 0\\
0 & I\end{bmatrix} \leq 
\begin{bmatrix}
X& 0\\
0 & -Y\end{bmatrix} 
\end{equation}
and
\begin{equation}
\begin{bmatrix}
X& 0\\
0 & -Y\end{bmatrix} \leq \begin{bmatrix}
Z & -Z\\
-Z & Z
\end{bmatrix}+
\frac{1}{\tau}\begin{bmatrix}
Z^2 & -Z^2\\
-Z^2 & Z^2
\end{bmatrix}.
\label{eq1:sum2}
\end{equation}
Here 
$$
Z=A\phi''(\delta)\hat \delta\otimes \hat \delta+A\frac{\phi'(\delta)}{\delta}\left(I-\hat \delta\otimes \hat \delta\right)
$$
and thus
$$
Z^2=A^2\left(\phi''(\delta)\right)^2\hat \delta\otimes \hat \delta+A^2\left(\frac{\phi'(\delta)}{\delta}\right)^2\left(I-\hat \delta\otimes \hat \delta\right).
$$
This implies in particular
\begin{equation}\label{eq:utest1}
(\alpha_1,a,X+BI)\in \overline{\mathcal{P}}_{Q_1}^{2,+}u\,(x,t),\quad (-\alpha_2,b,Y-BI)\in \overline{\mathcal{P}}_{Q_1}^{2,-}u\,(y,t),
\end{equation}
where
$$
a=A\phi'(\delta)\hat\delta+Bx, \quad b= A\phi'(\delta)\hat\delta-By.
$$
We now choose
$$
\tau = 4A\frac{\phi'(\delta)}{\delta} 
$$
\noindent {\bf Step 2: Basic estimates.}
Since 
$$
Z=A\left(\phi''(\delta) -\frac{\phi'(\delta)}{\delta}\right)\hat \delta\otimes \hat \delta+A\frac{\phi'(\delta)}{\delta}I,
$$
we see that 
\begin{equation}
\|Z\|\leq A\frac{\phi'(\delta)}{\delta}, \quad \|Z^2\|\leq A^2\left(|\phi''(\delta)|+\frac{\phi'(\delta)}{\delta}	\right)^2\leq 4A^2\left(\frac{\phi'(\delta)}{\delta}\right)^2,
\label{eq1:Zest}
\end{equation}
where we also used \eqref{eq:delta1}.

It will be convenient to introduce the notation 
$$
q=A\phi'(\delta)\hat\delta.
$$
Note that since $A>8$ and $B=4$
$$
|a|\leq |q|+4=2|q|+4-|q|\leq 2|q|
$$
where we used that by \eqref{eq:delta1} we have
$$
|q|=A|\phi'(\delta)|>8.
$$
Similarly
$$
|a|\geq |q|-4=|q|/2-4+A|\phi'(\delta)|/2\geq |q|/2.
$$
The same arguments can be carried out for $b$. Hence, 
\begin{equation}
\label{eq1:qest}
|q|/2\leq |a|\leq 2|q|, \quad |q|/2\leq |b|\leq 2|q|.
\end{equation}
By testing \eqref{eq1:sum1} and \eqref{eq1:sum2} with vectors of the form $(\xi,\xi)$ and $(\xi,0)$, where $\|\xi\| =1$ we obtain that 
\begin{equation}
\label{eq1:xynorm}
\|X\|,\|Y\|\leq \max(\|Z\|+\tau,\|Z\|+\frac{1}{\tau}\|Z^2\|)\leq 5A\frac{\phi'(\delta)}{\delta},\quad X-Y\leq 0,
\end{equation}
where we used \eqref{eq1:Zest} and that $|\phi''(\delta)|\leq \phi'(\delta)/\delta$.\\

\noindent {\bf Step 3: Using the equation.}
From the equation together with \eqref{eq:utest1} we obtain the two following inequalities
\begin{align}\label{eq:alphaineq}
|\alpha_1|^{p-2}\alpha_1&\leq |a|^{p-2}\tr(L(a)(X+BI)),\\\nonumber 
-|\alpha_2|^{p-2}\alpha_2&\geq |b|^{p-2}\tr(L(b)(Y-BI)),
\end{align}
where
$$
L(v)=I+(p-2)\frac{v\otimes v}{|v|^2}.
$$
Subtracting these inequalities, we obtain 
\begin{equation}
\label{eq1:viscdiffold}
|\alpha_1|^{p-2}\alpha_1+|\alpha_2|^{p-2}\alpha_2\leq |a|^{p-2}\tr(L(a)(X+BI))- |b|^{p-2}\tr(L(b)(Y-BI)).
\end{equation}
The aim is now to estimate the left hand side from below and the right hand side from above, and obtain a contradiction when choosing $A$ large enough. The idea is that there is at least one eigenvalue of $X-Y$ which is very negative when $A$ is large enough. This will violate an inequality obtained from the equation.\\

\noindent {\bf Step 4: Lower bound for the left hand side.}
 First of all, by \eqref{eq1:qest}, \eqref{eq1:xynorm} and \eqref{eq:alphaineq}
$$
|\alpha_i|^{p-2}\alpha_i\leq C|q|^{p-2}\left(A\frac{\phi'(\delta)}{\delta}+1\right)\leq C|q|^{p-2}A\frac{\phi'(\delta)}{\delta},\quad i=1,2, \quad C=C(n),
$$
where we used that $|q|=A\phi'(\delta)\geq 8$ and $\phi'(\delta)/\delta\geq e^2$ by \eqref{eq:delta1}, so that the constant can be absorbed. 
From the above together with relation 
$$\alpha_1+\alpha_2\geq Bt,
$$ 
it follows that\footnote{Recall the inequality
$$
(\beta_1+\beta_2)^{p-1}\leq 2^{p-2}\left(\beta_1^{p-1}+\beta_2^{p-1}\right),\quad \beta_1,\beta_2\geq 0,\quad p\geq 2.
$$}
$$
|\alpha_1|^{p-2}\alpha_1\geq -C-C\left(\alpha_2^+\right)^{p-1}\geq -C|q|^{p-2}\left(A\frac{\phi'(\delta)}{\delta}+1\right)\geq -C|q|^{p-2}A\frac{\phi'(\delta)}{\delta},
$$
where $C=C(p,n)$ and $\alpha_2^+$ is the positive part of $\alpha_2$. 
The same estimate holds also for $\alpha_2$. Thus
$$
|\alpha_i|^{p-2}\leq C|q|^{\frac{(p-2)^2}{p-1}}A^\frac{p-2}{p-1}\left(\frac{\phi'(\delta)}{\delta}\right)^\frac{p-2}{p-1},\quad i=1,2,\quad  C=C(p,n).
$$
This implies, via the inequality
$$
\Big||\beta_1+\beta_2|^{p-2}(\beta_1+\beta_2)-|\beta_1|^{p-2}\beta_1\Big|\leq (p-1)|\beta_2|(|\beta_1|+|\beta_2|)^{p-2}, 
$$
that
\begin{align}
|\alpha_1|^{p-2}\alpha_1+|\alpha_2|^{p-2}\alpha_2&\geq |\alpha_1|^{p-2}\alpha_1-|\alpha_1-Bt|^{p-2}(\alpha_1-Bt)\nonumber \\
&\geq -C(|\alpha_1|+1)^{p-2}\label{eq:alphaest}\\
&\geq -C_0|q|^{\frac{(p-2)^2}{p-1}}A^\frac{p-2}{p-1}\left(\frac{\phi'(\delta)}{\delta}\right)^\frac{p-2}{p-1}, \nonumber
\end{align}
where $C_0=C_0(p,n)$ and where we again absorbed the constant due to the bounds from below on $|q|$ and $\phi'(\delta)/\delta$. From \eqref{eq1:viscdiffold} and \eqref{eq:alphaest}, we can thus conclude
\begin{equation}
\label{eq1:viscdiff}
 |a|^{p-2}\tr(L(a)(X+BI))- |b|^{p-2}\tr(L(b)(Y-BI))\geq -C_0|q|^{\frac{(p-2)^2}{p-1}}A^\frac{p-2}{p-1}\left(\frac{\phi'(\delta)}{\delta}\right)^\frac{p-2}{p-1},
\end{equation}
where  $C_0=C_0(p,n)$.\\

\noindent {\bf Step 5: Upper bound for the right hand side.}
We now turn our attention to the right hand side. We split these terms into three parts
\begin{align*}
&|a|^{p-2}\tr(L(a)(X+BI))- |b|^{p-2}\tr(L(b)(Y-BI))\\
&=|b|^{p-2}\tr(L(b)(X-Y))+\tr((|a|^{p-2}L(a)-|b|^{p-2}L(b))X)\\
&+\left(|a|^{p-2}\tr(L(a) BI)+|b|^{p-2}\tr (L(b)BI)\right)\\
&=T_1+T_2+T_3.
\end{align*}

\noindent {\bf Step 5a: $T_1$.}

Testing inequality \eqref{eq1:sum2} with $(\hat\delta,-\hat \delta)$ we see that by \eqref{eq:delta1} and the choice of $\tau$
$$
\hat\delta\cdot (X-Y)\hat\delta\leq 4A\phi''(\delta)+\frac{4}{\tau}A^2\left(\phi''(\delta)\right)^2\leq 2A\phi''(\delta), 
$$
so that at least one of the eigenvalues of $X-Y$ is smaller than $2A\phi''(\delta)$. From \eqref{eq1:xynorm}, we know that the rest are non-positive. Hence, 
\begin{equation}
\label{eq1:T1est}
T_1\leq 2|b|^{p-2}A\phi''(\delta)\leq C_1A|q|^{p-2}\phi''(\delta)=-C_1A|q|^{p-2}\delta^{-1},\quad C_1=C_1(p),
\end{equation}
where we used \eqref{eq1:qest} and that the smallest eigenvalue of $L$ is $1$.\\

\noindent {\bf Step 5b: $T_2$.} For $T_2$ we have
\begin{equation}\label{eq1:T2est}
\begin{split}
T_2\leq n\|X\|\||a|^{p-2}L(a)-|b|^{p-2}L(b)\|&\leq  C\|X\||q|^{p-3}|B(x+y)|\\
&\leq C_2A\frac{\phi'(\delta)}{\delta}|q|^{p-3}\\
&=C_2|q|^{p-2}\delta^{-1},
\end{split}
\end{equation}
where $C_2=C_2(p,n)$,  and where we used the mean value theorem (for the mapping $v\mapsto |v|^{p-2}L(v)$), the definition of $q$, \eqref{eq1:qest}, \eqref{eq1:xynorm}, that $x,y\in B_1$ and that $B=4$. We also note that since 
$$
|a+s(b-a)|=|A\phi'(\delta)\hat\delta+Bx- sB(x+y)|\ge A\phi'(\delta)-B|x|-sB|x+y|\ge 8-3B\ge 2,
$$
for $s\in [0,1]$, the line between $a$ and $b$ does not pass through the origin.
\\

\noindent {\bf Step 5c: $T_3$.} For $T_3$ we have
\begin{equation}\label{eq1:T3est}
T_3\leq CB(p-1)n|q|^{p-2}\leq C_3|q|^{p-2},\quad C_3=C_3(p,n),
\end{equation}
where we have used \eqref{eq1:qest}. \\

\noindent {\bf Step 6: The contradiction.} Using \eqref{eq1:T1est}--\eqref{eq1:T3est} together with \eqref{eq1:viscdiff}, we obtain 
\begin{align*}
-C_0|q|^{\frac{(p-2)^2}{p-1}}A^\frac{p-2}{p-1}\left(\frac{\phi'(\delta)}{\delta}\right)^\frac{p-2}{p-1}\leq &-C_1A|q|^{p-2}\delta^{-1}+C_2|q|^{p-2}\delta^{-1}+C_3|q|^{p-2}\\
&=|q|^{p-2}\left(\frac{C_2-C_1A}{\delta}+C_3\right), 
\end{align*}
or equivalently
$$
C_0|q|^{\frac{(p-2)^2}{p-1}}A^\frac{p-2}{p-1}\left(\frac{\phi'(\delta)}{\delta}\right)^\frac{p-2}{p-1}+|q|^{p-2}\left(\frac{C_2-C_1A}{\delta}+C_3\right)\geq 0.
$$
This will be a contradiction if $A$ is chosen so that 
$$
|q|^{p-2}\left(\frac{C_2-C_1A/2}{\delta}+C_3\right)<0,\quad C_0|q|^{\frac{(p-2)^2}{p-1}}A^\frac{p-2}{p-1}\left(\frac{\phi'(\delta)}{\delta}\right)^\frac{p-2}{p-1}-\frac{C_1A|q|^{p-2}}{2\delta}<0.
$$
The first inequality is satisfied if we choose $A>2(C_3+C_2)/C_1$, which is a constant depending only on $n$ and $p$. Using that $|q|=A|\phi'(\delta)|$, the second inequality can be simplified to
$$
A>\frac{2C_0}{C_1}\delta^\frac{1}{p-1}, 
$$
so that it is sufficient to choose $A>2C_0/C_1$ which is a constant depending only on $n$ and $p$.
Hence, we arrive at a contradiction if 
$$
A>\max\left(2(C_3+C_2)/C_1, 2C_0/C_1\right).
$$
\end{proof}

\begin{cor}\label{cor:loglip} Suppose $u$ is a viscosity solution of \eqref{eq:art} in $Q_2$ such that $\osc_{Q_2} u\leq 1$. Then 
$$
|u(x,t)-u(y,t)|\leq C|x-y||\ln|x-y||
$$
for $t\in [-1,0]$ and $x,y\in B_\frac14$. Here $C=C(n,p)$.
\end{cor}
\begin{proof} First of all, by choosing $t=0$ and $x=0$ or $y=0$ in Proposition \ref{prop:loglip}, we obtain 
\begin{equation}\label{eq:logest1}
|u(x,0)-u(0,0)|\leq C|x-y||\ln |x-y||,\quad x\in B_1,\quad C=C(n,p).
\end{equation}
We now show how to obtain the desired regularity in the whole cylinder $B_{1/4}\times (-1,0)$. Let $(z_0,t_0)\in B_1\times (-1,0)$ and define
$$
v(x,t):=u\left(\frac{x}{2}+z_0,\frac{t}{2^\frac{p}{p-1}}+t_0 \right).
$$
Then $v$ is a solution of \eqref{eq:art} in $Q_2$. By construction, we also have
\[
\osc_{Q_2} v\leq \osc_{Q_2} u\leq  1.
\] 
We may therefore apply \eqref{eq:logest1} to $v$ and obtain
$$
\sup_{x\in B_r}|v(x,0)-v(0)|\leq C\,r|\ln r|,\quad 0<r<1.
$$
In terms of $u$ this implies
\begin{equation}
\label{eq:supest}
\sup_{x\in B_r(z_0)}|u(x,t_0)-u(z_0,t_0)|\leq Cr|\ln r|,\qquad 0<r<\frac{1}{2},
\end{equation}
upon renaming the constant. We note that this holds for any $z_0\in B_1,t_0\in (-1,0)$. Now take any pair $x,y\in B_{1/4}$ and set $|x-y|= r$. We observe that $r<1/2$ and we set $z=(x+y)/2$. Then we apply \eqref{eq:supest} with $z_0=z$ and obtain
\[
\begin{split}
|u(x,t_0)-u(y,t_0)|&\leq |u(x,t_0)-u(z,t_0)|+|u(y,t_0)-u(z,t_0)|\\
&\leq 2\sup_{w\in B_r(z)}|u(w,t_0)-u(z,t_0)|\\
&\leq 2\,C\,r\ln r=2\,C|x-y||\ln |x-y||.
\end{split}
\]
which is the desired result.
\end{proof}

\section{Lipschitz continuity}\label{sec:lip} We first prove some properties of the function $\varphi$ used in this section.
\begin{lem}\label{lem:phi2} Let 
$$
\varphi(r)=\begin{cases} r-r^\gamma, & 0<r\leq r_0= \left(\frac{1}{\gamma}\right)^\frac{1}{\gamma-1}\\ r_0-r_0^\gamma, &\text{otherwise}\end{cases},\quad \gamma\in \left(0,\min\left\{\frac{3}{2},\frac{p}{p-1}\right\}\right).
$$
Then 
$$
\varphi(r)<\left(\frac{1}{2\gamma}\right)^\frac{1}{\gamma-1}\left(1-\frac{1}{2\gamma}\right)
$$ implies
$$
r<\left(\frac{1}{2\gamma}\right)^\frac{1}{\gamma-1}<1,\quad  \varphi(r)=r-r^\gamma, \quad  \varphi'(r)\geq 1/2,\quad |\varphi''(r)|\leq \varphi'(r)/r.
$$
\end{lem}
\begin{proof}
First we note that $\varphi$ is non-decreasing. Moreover, 
$$
\varphi\left(\left(\frac{1}{2\gamma}\right)^\frac{1}{\gamma-1}\right)=\left(\frac{1}{2\gamma}\right)^\frac{1}{\gamma-1}\left(1-\frac{1}{2\gamma}\right).
$$
Therefore, $r<\left(\frac{1}{2\gamma}\right)^\frac{1}{\gamma-1}$ and by definition, $\varphi(r)=r-r^\gamma$. It is also straight forward to verify that
$$
\varphi'(r)=1-\gamma r^{\gamma-1}\geq \frac12
$$
whenever $r\leq \left(\frac{1}{2\gamma}\right)^\frac{1}{\gamma-1}$. Finally, 
$$
|\varphi''(r)|=\gamma(\gamma-1)r^{\gamma-2}\leq \varphi'(r)r^{-1}=r^{-1}-\gamma r^{\gamma-2}
$$
since $r^{\gamma-1}<1/(2\gamma)$ together with $\gamma<2$ implies $r^{\gamma-1}\leq \gamma^{-2}$.
\end{proof}
\begin{prop} Suppose $u$ is a viscosity solution of \eqref{eq:art} in $Q_2$ such that $\osc_{Q_2} u\leq 1$. Then 
$$
u(x,t)-u(y,t)\leq A\varphi (|x-y|)+\frac{B}{2}\left(|x|^2+|y|^2+t^2\right), 
$$
for $(x,t)\in {Q_1}$. Here
$$
\varphi(r)=\begin{cases} r-r^\gamma, & 0<r\leq r_0= \left(\frac{1}{\gamma}\right)^\frac{1}{\gamma-1}\\ r_0-r_0^\gamma, &\text{otherwise}\end{cases},\quad \gamma\in (1,\frac{p}{p-1})\cap (1,3/2)
$$
and $A=A(n,p)$ and $B$ is universal.
\end{prop}
\begin{proof} The proof is almost identical with the proof of Proposition \ref{prop:loglip}. The main differences are the different modulus of continuity and that we use the log-Lipschitz regularity in our estimates. We spell out the details. Let 
$$
\Phi(x,y,t)=u(x,t)-u(y,t)-A\varphi(|x-y|)-\frac{B}{2}\left(x^2+y^2+t^2\right).
$$
We will show that $\Phi(x,y,t)\leq 0$ for $t\in [-1,0]$ and $x,y\in B_1$. In order to do that we assume towards a contradiction that $\Phi$ has a positive maximum for $t\in [-1,0]$ and $x,y\in  \overline B_1$ at $(x,y,t)$. Since $\Phi(x,y,t)>0$ we have
\begin{equation}\label{eq:abeq}
A\varphi(|x-y|)+\frac{B}{2}\left(|x|^2+|y|^2+t^2\right)\leq |u(x,t)-u(y,t)|\leq 1.
\end{equation}
Therefore, by choosing  $B=33$ we can assure that $x,y\in B_{1/4}$ and $t\in (-1,0]$. Again, we let 
$$
\hat \delta=\frac{x-y}{|x-y|},\quad  \delta =|x-y|.
$$
By choosing 
$$
A>\frac{1}{\frac{1}{2\gamma}^\frac{1}{\gamma-1}\left(1-\frac{1}{2\gamma}\right)}
$$
estimate \eqref{eq:abeq} and Lemma \ref{lem:phi2} imply
\begin{equation}
\label{eq:delta2}
\delta <\left(\frac{1}{2\gamma}\right)^\frac{1}{\gamma-1}<1,\quad  \varphi(\delta)=\delta-\delta^\gamma, \quad  \varphi'(\delta)\geq 1/2,\quad |\varphi''(\delta)|\leq \varphi'(\delta)/\delta.
\end{equation}
From Corollary \ref{cor:loglip}, we know that $u$ is log-Lipschitz in $B_{1/4}\times (-1,0)$, and in particular $C^{2\gamma-2}$. We may therefore use \eqref{eq:abeq} to extract
$$
|x|^2+|y|^2+t^2\leq \frac{2}{B}|u(x,t)-u(y,t)|\leq \frac{2C}{B}|x-y|^{2\gamma-2}, \quad C=C(p,n)
$$
or
\begin{equation}
\label{eq:xyholder}
|x|,|y|,|t|\leq \sqrt{\frac{2C}{B}}|x-y|^{\gamma-1}.
\end{equation}

\noindent {\bf Step 1: Theorem of sums.}
From the parabolic theorem of sums (Theorem 8.3 in \cite{CIL92} and Theorem 9 in \cite{DFO11}) for any $\tau>0$, there are $X,Y\in S(n)$, $\alpha_1$ and $\alpha_2$ such that
\begin{align*}
(\alpha_1,A\varphi'(\delta)\hat\delta,X)\in \overline{\mathcal{P}}_{Q_1}^{2,+}\left(u-\frac{B}{2}|\cdot|^2\right)\,(x,t),\\
 (-\alpha_2,A\varphi'(\delta)\hat\delta,Y)\in \overline{\mathcal{P}}_{Q_1}^{2,-}\left(u+\frac{B}{2}|\cdot|^2\right)\,(y,t),
\end{align*}
$$
\alpha_1+\alpha_2 \geq  Bt,
$$
\begin{equation}
\label{eq:sum1}
-[\tau +||Z||]\begin{bmatrix}
I& 0\\
0 & I\end{bmatrix} \leq 
\begin{bmatrix}
X& 0\\
0 & -Y\end{bmatrix} 
\end{equation}
and
\begin{equation}
\begin{bmatrix}
X& 0\\
0 & -Y\end{bmatrix} \leq \begin{bmatrix}
Z & -Z\\
-Z & Z
\end{bmatrix}+
\frac{1}{\tau}\begin{bmatrix}
Z^2 & -Z^2\\
-Z^2 & Z^2
\end{bmatrix}.
\label{eq:sum2}
\end{equation}
Here
$$
Z=A\varphi''(\delta)\hat \delta\otimes \hat \delta+A\frac{\varphi'(\delta)}{\delta}\left(I-\hat \delta\otimes \hat \delta\right),
$$
$$
Z^2=A^2\left(\varphi''(\delta)\right)^2\hat \delta\otimes \hat \delta+A^2\left(\frac{\varphi'(\delta)}{\delta}\right)^2\left(I-\hat \delta\otimes \hat \delta\right)
$$
and we choose
$$
\tau = 4A\frac{\varphi'(\delta)}{\delta}. 
$$
This implies in particular
\begin{equation}\label{eq:utest2}
(\alpha_1,a,X+BI)\in \overline{\mathcal{P}}_{Q_1}^{2,+}u\,(x,t),\quad (-\alpha_2,b,Y-BI)\in \overline{\mathcal{P}}_{Q_1}^{2,-}u\,(y,t),
\end{equation}
where
$$
a=A\varphi'(\delta)\hat\delta+Bx, \quad b= A\varphi'(\delta)\hat\delta-By.
$$
\noindent {\bf Step 2: Basic estimates.}
Since 
$$
Z=A\left(\varphi''(\delta) -\frac{\varphi'(\delta)}{\delta}\right)\hat \delta\otimes \hat \delta+A\frac{\varphi'(\delta)}{\delta}I
$$
the last inequality in \eqref{eq:delta2} implies
\begin{equation}\label{eq:ZLest}
\|Z\|\leq A\frac{\varphi'(\delta)}{\delta}, \quad \|Z^2\|\leq 4A^2\left(\frac{\varphi'(\delta)}{\delta}	\right)^2
\end{equation}
We now introduce the notation
$$
q=A\varphi'(\delta)\hat\delta.
$$
By choosing $A\geq 200$ and using that $\varphi'(\delta)\geq 1/2$ (from \eqref{eq:delta2}), we may as in the proof of Proposition \ref{prop:loglip}, conclude
\begin{equation}
\label{eq:qLest}
|q|/2\leq a\leq 2|q|, \quad |q|/2\leq b\leq 2|q|.
\end{equation}
By testing \eqref{eq:sum1} and \eqref{eq:sum2} with vectors of the form $(\xi,\xi)$ and $(\xi,0)$, where $\|\xi\| =1$ we obtain that 
\begin{equation}
\label{eq:xynorm}
\|X\|,\|Y\|\leq \max(\|Z\|+\tau,\|Z\|+\frac{1}{\tau}\|Z^2\|)\leq 5A\frac{\varphi'(\delta)}{\delta},\quad X-Y\leq 0,
\end{equation}
where we used \eqref{eq:ZLest} and again that $|\varphi''(\delta)|\leq \varphi'(\delta)/\delta$.\\

\noindent {\bf Step 3: Using the equation.}
From the equation and \eqref{eq:utest2} we obtain the two following inequalities
\begin{align*}
|\alpha_1|^{p-2}\alpha_1&\leq |a|^{p-2}\tr(L(a)(X+BI)),\\
-|\alpha_2|^{p-2}\alpha_2&\geq |b|^{p-2}\tr(L(b)(Y-BI)),
\end{align*}
where
$$
L(v)=I+(p-2)\frac{v\otimes v}{|v|^2}.
$$
Subtracting these inequalities, we obtain
\begin{equation}
\label{eq:viscdiffold}
|\alpha_1|^{p-2}\alpha_1+|\alpha_2|^{p-2}\alpha_2\leq |a|^{p-2}\tr(L(a)(X+BI))- |b|^{p-2}\tr(L(b)(Y-BI)).
\end{equation}
We will now estimate the left hand side from below and the right hand side from above, and obtain a contradiction by choosing $A$ large enough.\\

\noindent {\bf Step 4: Lower bound for the left hand side.} The estimate of the left hand side is identical to the estimate done in Step 4 in the proof of Proposition \ref{prop:loglip}. This together with \eqref{eq:viscdiffold} yields
\begin{equation}
\label{eq:viscdiff}
|a|^{p-2}\tr(L(a)(X+BI))- |b|^{p-2}\tr(L(b)(Y-BI))\geq -C_0|q|^{\frac{(p-2)^2}{p-1}}A^\frac{p-2}{p-1}\left(\frac{\varphi'(\delta)}{\delta}\right)^\frac{p-2}{p-1},
\end{equation}
where $C_0=C_0(n,p)$.\\

\noindent {\bf Step 5: Upper bound for the right hand side.}
We split these terms into three parts
\begin{align*}
&|a|^{p-2}\tr(L(a)(X+BI))- |b|^{p-2}\tr(L(b)(Y-BI))\\
&=|b|^{p-2}\tr(L(b)(X-Y))\\
&+\tr((|a|^{p-2}L(a)-|b|^{p-2}L(b))X)\\
&+|a|^{p-2}\tr(L(a) BI)+|b|^{p-2}\tr (L(b)BI)\\
&=T_1+T_2+T_3.
\end{align*}
\noindent {\bf Step 5a: $T_1$.}
Testing inequality \eqref{eq:sum2} with $(\hat\delta,-\hat \delta)$, we see that by \eqref{eq:delta2} the choice of $\tau$
$$
\hat\delta\cdot (X-Y)\hat\delta\leq 4A\varphi''(\delta)+\frac{4}{\tau}A^2\left(\varphi''(\delta)\right)^2\leq 2A\varphi''(\delta), 
$$
so that at least one of the eigenvalues of $X-Y$ is smaller than $2A\varphi''(\delta)$. From \eqref{eq:xynorm}, we know that the rest are non-positive. Hence, 
\begin{equation}
\label{eq:T1est}
T_1\leq |b|^{p-2}A\varphi''(\delta)\leq CA|q|^{p-2}\varphi''(\delta)=-C_1A|q|^{p-2}\delta^{\gamma-2},\quad C_1=C_1(p),
\end{equation}
where we used \eqref{eq:qLest} and that the smallest eigenvalue of $L$ is $1$.\\

\noindent {\bf Step 5b: $T_2$.}
For $T_2$ we have
\begin{equation}\label{eq:T2est}
\begin{split}
T_2\leq n\|X\|\||a|^{p-2}L(a)-|b|^{p-2}L(b)\|&\leq  C\|X\||q|^{p-3}|B(x+y)|\\
&\leq C'A\frac{\varphi'(\delta)}{\delta}|q|^{p-3}\delta^{\gamma-1}\\
&=C_2|q|^{p-2}\delta^{\gamma-2},
\end{split}
\end{equation}
where $C_2=C_2(p,n)$, and where we used the mean value theorem, the definition of $q$, \eqref{eq:qLest}, \eqref{eq:xynorm} and \eqref{eq:xyholder}. We also note that since 
$$
|a+s(b-a)|=|A\varphi'(\delta)\hat\delta+Bx- sB(x+y)|\ge A\varphi'(\delta)-B|x|-sB|x+y|\ge 100-3B\ge 1,
$$
for $s\in [0,1]$, the line between $a$ and $b$ does not pass through the origin.\\

\noindent {\bf Step 5c: $T_3$.}
For $T_3$ we have
\begin{equation}\label{eq:T3est}
T_3\leq BC(p-1)n|q|^{p-2}\leq C_3|q|^{p-2},\quad C_3=C_3(p,n), 
\end{equation}
where we used \eqref{eq:qLest}. \\

\noindent {\bf Step 6: The contradiction.}
Using \eqref{eq:T1est}--\eqref{eq:T3est} together with \eqref{eq:viscdiff}, we obtain 
$$
-C_0|q|^{\frac{(p-2)^2}{p-1}}A^\frac{p-2}{p-1}\left(\frac{\varphi'(\delta)}{\delta}\right)^\frac{p-2}{p-1}\leq -C_1A|q|^{p-2}\delta^{\gamma-2}+C_2|q|^{p-2}\delta^{\gamma-2}+C_3|q|^{p-2}
$$
or 
$$
0\leq |q|^{p-2}\left((C_2-C_1A)\delta^{\gamma-2}+C_3\right)+C_0|q|^{\frac{(p-2)^2}{p-1}}A^\frac{p-2}{p-1}\left(\frac{\varphi'(\delta)}{\delta}\right)^\frac{p-2}{p-1}.
$$
This is a contradiction if we choose $A$ such that 
$$
0> (C_2-C_1A/2)\delta^{\gamma-2}+C_3,\quad 0> -C_1A/2|q|^{p-2}\delta^{\gamma-2} + C_0|q|^{\frac{(p-2)^2}{p-1}}A^\frac{p-2}{p-1}\left(\frac{\varphi'(\delta)}{\delta}\right)^\frac{p-2}{p-1}
$$
The first inequality holds if we choose $A>2(C_3+C_2)/C_1$ and the second inequality is equivalent to 
$$
A>\frac{2C_0}{C_1}\delta^{\frac{p}{p-1}-\gamma}.
$$
once we recall $|q|=A\varphi'(\delta)$. Since $\delta <(1/(2\gamma))^\frac{1}{\gamma-1}<1$ and $\gamma<p/(p-1)$, it is therefore sufficient to choose 
$$
A>\frac{2C_0}{C_1} 
$$
in order to have the second inequality. All in all, we arrive at a contradiction by choosing
$$
A>\max\left(\frac{2C_0}{C_1},2(C_3+C_2)/C_1\right),
$$
which is a constant depending only on $n$ and $p$.
\end{proof}
That the result above implies the local Lipschitz regularity can be proved exactly as Corollary \ref{cor:loglip}.
\begin{cor}\label{cor:lip} Suppose $u$ is a viscosity solution of \eqref{eq:art} in $Q_2$ such that $\osc_{Q_2} u\leq 1$. Then 
$$
|u(x,t)-u(y,t)|\leq C|x-y|,
$$
for $t\in [-1,0]$ and $x,y\in B_\frac14$. Here $C=C(n,p)$.
\end{cor}

\begin{rem}\label{rem:flex} By a simple covering argument we may also obtain an estimate
$$
|u(x,t)-u(y,t)|\leq C|x-y|,\quad C=C(n,p)
$$
for $(x,t)\in Q_1$, for a solution $u$ in $Q_2$ such that $\osc_{Q_2} u\leq 1$.  

Indeed, we can cover $\overline{Q_1}$ with finitely many cylinders of the form $B_{1/8}(x_i)\times (t_i-1/(2^{p/(p-1}),t_i)$ where $x_i\in B_1$ and $t_i\in (-1,0)$. Corollary \ref{cor:lip} applied to the functions
$$
v_i(x,t)=u(x/2+x_i,1/(2^{p/(p-1})t+t_i),
$$
which are all solutions in $Q_2$, implies
$$
|v_i(x,t)-v_i(y,t)|\leq C|x-y|, \quad x,y\in B_{1/4}, t\in (-1,0).
$$
Going back to $u$ this implies
$$
|u(x,t)-u(y,t)|\leq C|x-y|, \quad x,y\in B_{1/8}(x_i), t\in (t_i-1/(2^{p/(p-1}),t_i),
$$
for any $i$, which implies the desired estimate.
\end{rem}
\section{H\"older regularity in time}\label{sec:time}
In this section we prove H\"older estimates in the $t$-variable. It amounts to constructing a suitable supersolution. See Lemma 3.1 in \cite{IJS16} or Lemma 9.1 in \cite{BBL02} for similar results.
\begin{prop}\label{prop:timereg}
Suppose $u$ is a viscosity solution of \eqref{eq:art} in $Q_2$ such that $\|u\|_{L^\infty(Q_2)}\leq 1$. Then 
$$
|u(0,t)-u(0,s)|\leq C|t-s|^\frac{p-1}{p},
$$
for $t,s\in [-1,0]$.
\end{prop}
\begin{proof} Fix $t_0\in (-1,0)$. We claim that the following estimate holds
\begin{equation}\label{tineq}
u(x,t)-u(0,t_0)\leq\phi(t,x):= \eta +A(t-t_0)+B|x|^\frac{p}{p-1}, 
\end{equation}
for $t\in [t_0,0]$, $x\in  B_1$, whenever $A$, $B$ and $\eta$ satisify
\begin{equation}
\label{eq:AB}
A=\left(\frac{p}{p-1}\right)n^\frac{1}{p-1}B, \quad B^{p-1}=\max\left( \frac{1}{p}\left(\frac{p-1}{p}\right)^{p-1}\frac{\|\nabla u\|_{L^\infty(Q_1)}^p}{\eta},2^{p-1}\right),
\end{equation}
This is accomplished by making $\phi$ a supersolution and applying the comparison principle.

We first remark that for $x\in \partial B_1$, \eqref{tineq} reads
$$
u(x,t)-u(0,t_0)\leq \eta+A(t-t_0)+B, 
$$
which clearly holds if $B\geq 2$. In addition, when $t=t_0$ \eqref{tineq} reduces to
\begin{equation}
\label{eq:parabola1}
u(x,t_0)-u(0,t_0)\leq \eta+B|x|^\frac{p}{p-1}.
\end{equation}
By Corollary \ref{cor:lip} and Remark \ref{rem:flex}, we know that $u$ is Lipschitz in space in $Q_1$. Thus
$$
|u(x,t_0)-u(0,t_0)|\leq \|\nabla u\|_{L^\infty(Q_1)}|x|,\quad \|\nabla u\|_{L^\infty(Q_1)}<C(n,p).
$$
Hence, \eqref{eq:parabola1} is valid if
$$
\|\nabla u\|_{L^\infty(Q_1)}|x|\leq \eta+B|x|^\frac{p}{p-1}
$$
which holds if\footnote{Find the min of this radial function.}
$$
B^{p-1}\geq \frac{1}{p}\left(\frac{p-1}{p}\right)^{p-1}\frac{\|\nabla u\|_{L^\infty(Q_1)}^p}{\eta},
$$
which holds due to \eqref{eq:AB}. We have thus settled that \eqref{tineq} holds on the parabolic boundary of $B_1\times [t_0,0]$. We now see that 
$$
|\partial_t \phi|^{p-2}\partial_t \phi-\Delta_p \phi = A^{p-1}-\left(\frac{p}{p-1}\right)^{p-1}B^{p-1}n\geq 0
$$
since $A= \left(\frac{p}{p-1}\right)n^\frac{1}{p-1}B$. Therefore, $\phi$ is a supersolution and \eqref{tineq} holds in $B_1\times [t_0,0]$ by the comparison principle (Proposition \ref{UsualComparison}), given that \eqref{eq:AB} is satisfied.

In order to prove the assertion, we choose $\eta = \|\nabla u\|_{L^\infty(Q_1)}|t-t_0|^\frac{p-1}{p}$. We note that \eqref{eq:AB} implies with this choice of $\eta$ that
$$
B\leq  \frac{1}{p^\frac{1}{p-1}}\frac{p-1}{p}\frac{\|\nabla u\|_{L^\infty(Q_1)}^\frac{p}{p-1}}{\eta^\frac{1}{p-1}}+2^{p-1}\leq \frac{1}{p^\frac{1}{p-1}}\frac{p-1}{p}\frac{\|\nabla u\|_{L^\infty(Q_1)}}{|t-t_0|^\frac{1}{p}}+2^{p-1}.
$$
Then \eqref{tineq} and \eqref{eq:AB} imply
\begin{align*}
u(0,t)-u(0,t_0)&\leq \|\nabla u\|_{L^\infty(Q_1)}|t-t_0|^\frac{p-1}{p}+\left(\frac{p}{p-1}\right)n^\frac{1}{p-1}B(t-t_0)\\
&\leq C(1+\|\nabla u\|_{L^\infty(Q_1)}+|t-t_0|^\frac{1}{p})|t-t_0|^\frac{p-1}{p}, \\
&\leq C|t-t_0|^\frac{p-1}{p},\quad C=C(n,p).
\end{align*}
Since this holds for $u$ and $-u$, the reverse inequality also holds.
\end{proof}
\begin{cor}\label{cor:timereg} Suppose $u$ is a viscosity solution of \eqref{eq:art} in $Q_2$ such that $\osc_{Q_2} u\leq 1$. Then 
$$
|u(x,t)-u(x,s)|\leq C|t-s|^\frac{p-1}{p},
$$
for $t,s\in [-1/2^\frac{p}{p-1},0]$ and $x,y\in B_1$. Here $C=C(n,p)$.
\end{cor}
\begin{proof}
Let $z_0\in B_1$ and define
$$
v(x,t):=u\left(\frac{x}{2}+z_0,\frac{t}{2^\frac{p}{p-1}}\right).
$$
Then $v$ is a solution of \eqref{eq:art} in $Q_2$. By construction, we also have
\[
\osc_{Q_2} v\leq 1.
\] 
We may therefore apply Proposition \ref{prop:timereg} to $v$ and obtain
$$
|v(0,t)-v(0,s)|\leq C|s-t|^\frac{p-1}{p}, \quad s,t\in [-1,0].
$$
In terms of $u$ this implies
$$
|u(z_0,\frac{t}{2^\frac{p}{p-1}})-v(z_0,\frac{s}{2^\frac{p}{p-1}})|\leq C|s-t|^\frac{p-1}{p}, \quad s,t\in [-1,0]
$$
which is the desired result, upon renaming $C$.
\end{proof}
\section{Proof of the regularity theorem}\label{sec:main}
We have now everything needed for the proof Theorem \ref{thm:main}.

\begin{proof}[~Proof of Theorem \ref{thm:main}] Define 
$$
u_R(x,t)=\frac{u(Rx+x_0,R^\frac{p}{p-1} t+t_0)}{2\|u\|_{L^\infty(Q_{2R}(x_0,t_0))}}.
$$
Then $u_R$ solves \eqref{eq:art} in $Q_2$ and $\osc_{Q_2} u\leq 1$. From Remark \ref{rem:flex} and Corollary \ref{cor:timereg} we obtain 
$$
|u_R(x,t)-u_R(y,t)|\leq C|x-y|,\quad |u_R(x,t)-u_R(x,s)|\leq C|t-s|^\frac{p-1}{p},
$$
for $x,y\in B_1, s,t\in (-1/2^\frac{p}{p-1},0)$. Coming back to $u$, this means
$$
|u(x,t)-u(y,t)|\leq C\|u\|_{L^\infty(Q_{2R}(x_0,t_0))}\frac{|x-y|}{R}\\
$$
and
$$
|u(x,t)-u(x,s)|\leq C\|u\|_{L^\infty(Q_{2R}(x_0,t_0))}\frac{|t-s|^\frac{p-1}{p}}{R}
$$
for all $x,y\in B_R(x_0)$ and $s,t\in (t_0-(R/2)^\frac{p}{p-1},t_0)$. The desired result now follows from the triangle inequality.
\end{proof}

\section{The large time behavior}\label{sec:largetime}
In \cite{HL16}, the unique viscosity solution of \eqref{pParabolic} is constructed. It is proved that this is also a weak solution. In addition, the large time behavior of weak solutions (which thus also applies to viscosity solutions) is characterized:
\begin{thm}\label{LargeTthm} Assume $g\in W^{1,p}_0(\Omega)$. Then for any weak solution $u$ of \eqref{pParabolic}, the limit 
$$
w:=\lim_{t\rightarrow \infty}e^{\lambda_p^\frac{1}{p-1} t}u(\cdot,t)
$$ exists in $W^{1,p}_0(\Omega)$ and is a $p$-ground state, provided $w\not\equiv 0$. In this case, $u(\cdot, t)\not\equiv 0$ for $t\ge 0$ and 
$$
\lambda_p=\lim_{t\rightarrow \infty}\frac{\displaystyle\int_\Omega|\nabla u(x,t)|^pdx}{\displaystyle\int_\Omega|u(x,t)|^pdx}. 
$$
\end{thm} 
We now have all the ingredients needed for the proof of Theorem \ref{thm:largetime}.
\begin{proof}[~Proof of Theorem \ref{thm:largetime}] Let $\tau_k$ be an increasing sequence of positive numbers such that $\tau_k\to \infty$ as $k\rightarrow\infty$. Since any viscosity solution is also a weak solution, Theorem \ref{LargeTthm} establishes that 
\begin{equation}\label{tauconv}
\lim_{k\rightarrow \infty}e^{\lambda_p^\frac{1}{p-1} \tau_k}v(\cdot,\tau_k)=w,
\end{equation}
in $W_0^{1,p}(\Omega)$. We also know that $w$ does not depend on the sequence $\tau_k$. It is therefore enough to prove that this seqeuence has a subseqence that convergences uniformly to $w$ on $\overline{\Omega}$. 
Define 
$$
v^k(x,t)=e^{\lambda_p^\frac{1}{p-1} \tau_k}v(x,t+\tau_k).
$$
We remark that $e^{-\lambda_p^\frac{1}{p-1} t}\phi $ is a solution of equation \eqref{eq:art}.
By the comparison principle, 
\begin{equation}\label{vkest}
|v^k(x,t)|\leq \phi(x)
\end{equation}
for $(x,t)\in \Omega\times [-1,1]$ for all $k\in \N$ large enough. These bounds together with Theorem \ref{thm:main} give that $v^k$ is uniformly bounded in $C^\alpha(B\times [0,1])$ for any ball $B\subset\subset \Omega$ for $\alpha\leq (p-1)/p$. By a routine covering argument, $v^k$ is then uniformly bounded in $C^\alpha(K\times [0,1])$ for any compact $K\subset\subset \Omega$. 
Since $\phi$ is continuous up to the boundary of $\Omega$, \eqref{vkest} together with these local estimates implies that the sequence $v^k$ is equicontinuous on $\overline\Omega \times [0,1]$ (see for instance the proof of Theorem 1.3 in \cite{HL16b} for details). By the Arzel\`a-Ascoli theorem, we can extract a subsequence $v^{k_j}$ such that  
$$
v^{k_j}\to e^{-\lambda_p^\frac{1}{p-1}t}w,
$$
uniformly on $\overline\Omega \times [0,1]$. Letting $t=0$, this establishes the desired existence of a uniformly convergent subsequence.
\end{proof}

\bibliographystyle{amsrefs}
\bibliography{ref.bib}
\end{document}